\documentclass[article, 11pt]{amsart}
\usepackage{amsfonts,amsthm,amsbsy,amsmath,amssymb,verbatim}
\usepackage{latexsym,enumerate,enumitem}
\usepackage{mathrsfs} 
\usepackage{tikz-cd}
\usetikzlibrary{positioning}
\usepackage{xcolor}
\usepackage{setspace}
\usepackage{geometry}
\usepackage{hyperref}
	\hypersetup{colorlinks, breaklinks,
            	linkcolor = blue,
				urlcolor = blue,
				anchorcolor = blue,
				citecolor = blue}

\usepackage{cleveref}

\usepackage{caption}
\usepackage{subcaption}

\newtheorem{theorem}{Theorem}

\newtheorem{prop}[theorem]{Proposition}

\newtheorem{claim}[theorem]{Claim}

\newtheorem*{theorem*}{Theorem}
\newtheorem*{lemma*}{Lemma}
\newtheorem*{prop*}{Proposition}
\newtheorem*{corollary*}{Corollary}
\newtheorem*{remark*}{Remark} 
\newtheorem*{remarks*}{Remarks}
\newtheorem*{conj*}{Conjecture}

\def\N{{\mathbb N}}

\def\S{\mathbb{S}}

\def\R{{\mathbb R}}
\def\Z{{\mathbb Z}}

\def\B{{\mathbb B}}

\def\d{\delta}

\def\s{\sigma}

\renewcommand{\bar}{\overline}

\DeclareFontFamily{U}{mathx}{\hyphenchar\font45}
\DeclareFontShape{U}{mathx}{m}{n}{
	<5> <6> <7> <8> <9> <10>
	<10.95> <12> <14.4> <17.28> <20.74> <24.88>
	mathx10
}{}

\def\wh{\widehat}

\newtheorem{innercustomgeneric}{\customgenericname}
\providecommand{\customgenericname}{}
\newcommand{\newcustomtheorem}[2]{%
	\newenvironment{#1}[1]
	{%
		\renewcommand\customgenericname{#2}%
		\renewcommand\theinnercustomgeneric{##1}%
		\innercustomgeneric
	}
	{\endinnercustomgeneric}
}

\newcustomtheorem{customthm}{Theorem}
\newcustomtheorem{customlemma}{Lemma}
\newcustomtheorem{customprop}{Proposition}

\newcommand{\MLp}{L^{p_1}(\R^n)\times\cdots\times L^{p_m}(\R^n)\to L^p(\R^n)}

\begin{document}

 \title[Multilinear singular integrals with homogeneous kernels near $L^1$]{Multilinear singular integrals with homogeneous kernels near $L^1$}
\author[G.~Dosidis and L.~Slav{\'i}kov{\'a}]{Georgios~Dosidis and Lenka~Slav{\'i}kov{\'a}}

\newcommand{\Addresses}{{
		\bigskip
		\footnotesize
		
		\textsc{Georgios~Dosidis.}
		\textsc{Department of Mathematical Analysis, Faculty of Mathematics and Physics, Charles University, Sokolovsk\'a 83, 186 75 Praha 8, Czech Republic}\par\nopagebreak
		\textit{E-mail address:} \texttt{Dosidis@karlin.mff.cuni.cz}
		
		\textsc{Lenka~Slav{\'i}kov{\'a}.}
		\textsc{Department of Mathematical Analysis, Faculty of Mathematics and Physics, Charles University, Sokolovsk\'a 83, 186 75 Praha 8, Czech Republic}\par\nopagebreak
		\textit{E-mail address:} \texttt{slavikova@karlin.mff.cuni.cz}

}}
\thanks{The authors were supported by the Primus research programme PRIMUS/21/SCI/002 of Charles University.}

\begin{abstract}  We obtain the optimal open range of $\MLp$ bounds for multilinear singular integral operators with homogeneous kernels of the form $\Omega(\frac{y}{|y|})|y|^{-mn}$, where $\Omega$ is a function in $L^{q}(\S^{mn-1})$ with vanishing integral and $q>1$.\end{abstract}

\maketitle

\section{Introduction}
Singular integral operators of the form
\begin{equation}\label{E:linear}
Tf(x) := pv \int_{\R^{n}} f(y)K(x-y)dy
\end{equation}
have been extensively studied ever since their introduction by Calder\'on and Zygmund in \cite{CZ52}. In the special case of homogeneous kernels $K(x)= \frac{\Omega(x/|x|)}{|x|^{n}}$, which includes the Hilbert and the Riesz transforms, boundedness of the operator $T$ from $L^p(\R^n)$ to itself can be obtained for all $1<p<\infty$ without any smoothness assumption on $\Omega$. This was first proven in \cite{CZ56} under the assumption that $\Omega$ has vanishing integral on the unit sphere and belongs to the Orlicz space $L\log L(\S^{n-1})$, and subsequently improved in \cite{CW77,C79} by making use of Hardy spaces. Endpoint bounds were obtained in \cite{C88,CRdF88,H88,S96}.

In this paper, we investigate the multilinear variant of the operator~\eqref{E:linear}. Namely, given $m\in \N$ and a function $\Omega$ with zero mean on the unit sphere which, in addition, belongs to $L^q(\S^{mn-1})$ for some $q>1$, we set $K(x)= \dfrac{\Omega(x')}{|x|^{mn}}$ for $x \in \R^{mn} \setminus \{0\}$, where $x' = \frac{x}{|x|}$, and consider the multilinear singular integral operator 
\begin{equation}\label{eq:T}T^m_{\Omega}(f_1,f_2,\dots,f_m)(x) := pv \int_{\R^{mn}} K(x-y_1,x-y_2,\dots,x-y_m) \prod_{i=1}^m f_i(y_i) dy.
\end{equation}
Throughout the paper, we assume that exponents $p_1,\dots,p_m, p$ satisfy the H\"older scaling $\frac1p = \sum_{i=1}^m\frac{1}{p_i}$, and we study the question of $\MLp$ boundedness of $T^m_\Omega$.

The $L^{p_1}(\R^n) \times L^{p_2}(\R^n) \to L^p(\R^n)$ boundedness of the bilinear operator $T^2_\Omega$ was proven by Coifman and Meyer in \cite{CM75} under the assumption that $\Omega$ is a function of bounded variation and $1<p_1,p_2,p<\infty$. This result led, in particular, to an alternative proof of $L^p$ estimates for Calder\'on's first commutator, an operator introduced by Calder\'on in~\cite{C65} in connection with the study of elliptic partial differential equations. A yet another proof of the boundedness of Calder\'on's commutator relevant for our present approach was given by Muscalu in~\cite{M14}.
An extension of the Coifman--Meyer theorem~\cite{CM75} to arbitrary exponents satisfying $1<p_1,p_2<\infty$, with $p$ being possibly less than $1$, was obtained by Kenig and Stein in~\cite{KS99}  and  by Grafakos and Torres in \cite{GT99}. 

The first result concerning bounds for $T^2_\Omega$ in the case when $\Omega$ does not possess any smoothness was obtained in  \cite{DGHST11}. In that paper, the authors considered the one-dimensional case $n=1$ and proved boundedness of the operator $T^2_\Omega$ in a certain range of exponents under the assumption that the even part of $\Omega$ belongs to the Hardy space $H^1(\S^1)$. The higher--dimensional case was discussed in \cite{GHH18}, where boundedness of $T^2_\Omega$ in the full range $1<p_1,p_2<\infty$ was proven under the assumption that $\Omega\in L^\infty(\S^{2n-1})$, and boundedness in the \emph{``local $L^2$ case"}, that is $2\leq p_1,p_2\leq\infty$ and $1\leq p \leq 2$, was proven for $\Omega\in L^2(\S^{2n-1})$. The local $L^2$ case bound was later improved in \cite{GHS20}, weakening the $L^2$ assumption on $\Omega$ to the requirement that $\Omega \in L^q(\S^{2n-1})$ with $q>4/3$. Boundedness in the extended range of exponents $\frac1p+\frac1q<2$, when $q>4/3$, was proven recently in \cite{HP21}.

In this work, we extend the boundedness results for $T^2_\Omega$ to the case $1<q \leq 4/3$. Let $d\sigma$ denote the surface measure on the sphere.
\begin{theorem}\label{th:1} Let $1<p_1,p_2<\infty$, $\frac{1}{p}=\frac{1}{p_1}+\frac{1}{p_2}$, and suppose that $\Omega \in L^q(\S^{2n-1})$, $q>1$, with $\int_{\S^{2n-1}} \Omega(\theta) d\sigma(\theta)= 0$. If 
\begin{equation}\label{E:opt-cond}
\frac{1}{p}+\frac{1}{q}<2,
\end{equation}
 then there exists a constant $C=C(n,p_1,p_2,q)$ such that
	\begin{equation}\label{E:strong-type-estimate}
	\|T^2_{\Omega}(f_1,f_2)\|_{L^p(\R^n)}\leq C { \|\Omega\|_{L^q(\S^{2n-1})}} \|f_1\|_{L^{p_1}(\R^n)}\|f_2\|_{L^{p_2}(\R^n)}.
	\end{equation}
\end{theorem}

The proof of \Cref{th:1} relies on estimates obtained in \cite{GHH18, HP21} combined with a new estimate of minimal blow-up when $\Omega\in L^1(\S^{2n-1})$.

\begin{remark*} The question of weak--type boundedness of $T^2_\Omega$ in the limiting case $\frac{1}{p}+\frac1q = 2$ remains open. In addition, no bounds are known for the operator $T^2_\Omega$ when $\Omega$ belongs merely to $L\log L(\S^{2n-1})$ except in the one-dimensional setting, in which a partial result was obtained in~\cite{DGHST11}.
\end{remark*}

We graph the range of boundedness of $T^2_\Omega$. The dotted line corresponds to $\frac{1}p + \frac1q =2$. Here and in what follows, $q'$ stands for the exponent satisfying $\frac{1}{q}+ \frac{1}{q'}=1$.

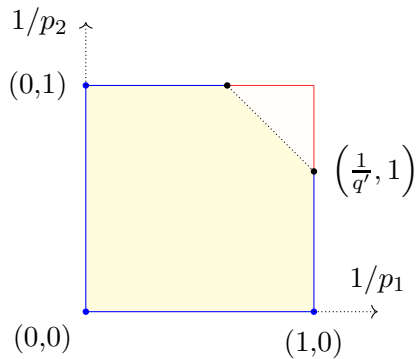
\begin{figure}[h]
	\newcommand{\xx}{3}
	\newcommand{\yy}{3}
	
	\begin{center}
		\begin{tikzpicture}
			\coordinate (O) at (0,0);
			\coordinate (A) at (\xx,0);
			\coordinate (B) at (0,\yy);
			\coordinate (C) at (\xx,\yy);
			\coordinate (s1) at (\xx,0.62*\yy);
			\coordinate (s2) at (0.62*\xx,\yy);
			\coordinate (X) at (1.28*\xx,0);
			\coordinate (Y) at (0,1.28*\yy);
			
			\draw[white,fill=yellow!16] (O) -- (A) -- (s1) -- (s2) -- (B) -- (O) -- cycle;
			\draw[white,fill=yellow!02.5] (s1) -- (C) -- (s2) -- (s1) -- cycle;
			
			\draw[blue] (O) -- (A);
			\draw[blue] (O) -- (B);
			\draw[blue] (A) -- (s1);
			\draw[blue] (B) -- (s2);
			\draw[red!80] (s1) -- (C);
			\draw[red!80] (s2) -- (C);
			\draw[black,densely dotted] (s1) -- (s2);

			\foreach \rr in {O, A, B} {\filldraw[blue](\rr) circle(1pt);}
			\foreach \rr in {s1,s2} {\filldraw[black](\rr) circle(1pt);}
			\node [below  =1mm of A]  {(1,0)};
			\node [left  =1mm of B]  {(0,1)};
			\node [right =1mm of s1]  {$\left(\frac{1}{q'} , 1\right)$};
			\node [below left = 0.5mm of O]  {(0,0)};
			
			\node [above  =1mm of X]  {$1/p_1$};
			\node [left =1mm of Y] {$1/p_2$};
			\draw[->,black,densely dotted] (A) -- (X);
			\draw[->,black,densely dotted] (B) -- (Y);

		\end{tikzpicture}
		\caption[Figure 1]{Optimal range of $L^{p_1}(\R^n)\times L^{p_2}(\R^n)\to L^p(\R^n)$ boundedness of $T^2_\Omega$ when $\Omega\in L^q(\S^{2n-1})$, $q>1$. }\label{F1}
	\end{center}
\end{figure}

The boundedness of the multilinear operator $T^m_\Omega$ with a rough kernel was first treated in \cite{GHHP20}, where initial $L^{2}(\R^n)\times\cdots\times L^{2}(\R^n)\to L^{2/m}(\R^n)$ bounds were established. In order to discuss the boundedness region for this  operator, we introduce some notation. For any multiindex $\alpha\in \{0,1\}^m$, we write $|\alpha|=\sum_{i=1}^m\alpha_i$. If $\vec{\frac{1}{p}}  = \left(\frac{1}{p_1},\dots, \frac{1}{p_m}\right)$, we set $\frac{1}{p_\alpha}  = \alpha\cdot \vec{\frac{1}{p}} = \sum_{i=1}^m \frac{\alpha_i}{p_i}$. Then $\frac{1}{p} = \frac{1}{p_{(1,\dots,1)}} =\sum_{i=1}^m \frac{1}{p_i}$. For any $q>1$, we write $\vec{\frac{1}{p}}\in\mathcal{H}^m(q)$ if for all $\alpha\in\{0,1\}^m$,
\begin{equation}\label{eq:H^m}
	\frac{1}{p_\alpha} + \frac{|\alpha| - 1}q < |\alpha|.	
\end{equation}
When $|\alpha|=0$ or $1$, \Cref{eq:H^m} is satisfied for all $1<p_1,\dots,p_m< \infty$, while the condition for $|\alpha|=m$ is
\begin{equation*}
	\frac{1}{p} + \frac{m - 1}q < m.	
\end{equation*}
%
%
%
We graph $\mathcal{H}^3(q)$ in the \Cref{fig:2} below. The equation for $|\alpha|=3$ corresponds to the red triangle, while the $3$ equations for $|\alpha|=2$, which have the form
\[
\frac{1}{p_1}+\frac{1}{p_2} +\frac{1}{q} <2, \quad \frac{1}{p_1}+\frac{1}{p_3} +\frac{1}{q} <2, \quad \frac{1}{p_2}+\frac{1}{p_3} +\frac{1}{q} <2,
\] 
correspond to the three red rectangles.

Recently in \cite{GHHP22}, the bounds for $T^m_\Omega$ were extended from the middle point $\left(\frac12,\dots,\frac12\right)$ to the whole region $\mathcal{H}^m(q)$, when $q\geq 2$. Similar to the bilinear case, we extend these boundedness results to the case $1<q<2$. 

\begin{theorem}\label{th:2} Let $1<p_1,p_2,\dots,p_m<\infty$, $\frac{1}{p} = \sum_{i=1}^m \frac{1}{p_i}$ be such that $\left(\frac{1}{p_1},\dots,\frac{1}{p_m}\right)\in \mathcal{H}^m(q)$ and suppose that  $\Omega\in L^q(\S^{mn-1})$, $q>1$, with $\int_{\S^{mn-1}} \Omega(\theta)d\sigma(\theta) = 0$. Then there exists a constant $C=C(n,p_1,p_2,\dots,p_m,q)$ such that
	\begin{equation}\label{E:strong-type-estimate-M}
		\|T^m_{\Omega}(f_1,\dots,f_m)\|_{L^p(\R^n)}\leq C { \|\Omega\|_{L^q(\S^{mn-1})}} \prod_{i=1}^m \|f_i\|_{L^{p_i}(\R^n)}.
	\end{equation}
\end{theorem}

In particular, if $1<p,p_1,p_2,\dots,p_m<\infty$, then \Cref{E:strong-type-estimate-M} holds for all $q>1$. Note that for $m=2$, the range $\mathcal H^2(q)$ corresponds to the condition $\frac1p+\frac1q < 2$, so \Cref{th:1} is a special case of \Cref{th:2}. Moreover, for $m=1$, the condition in \Cref{eq:H^m} becomes $p>1$, which is the range of boundedness for the linear operator when $\Omega\in L^q(\S^{n-1})$ for some $q>1$.

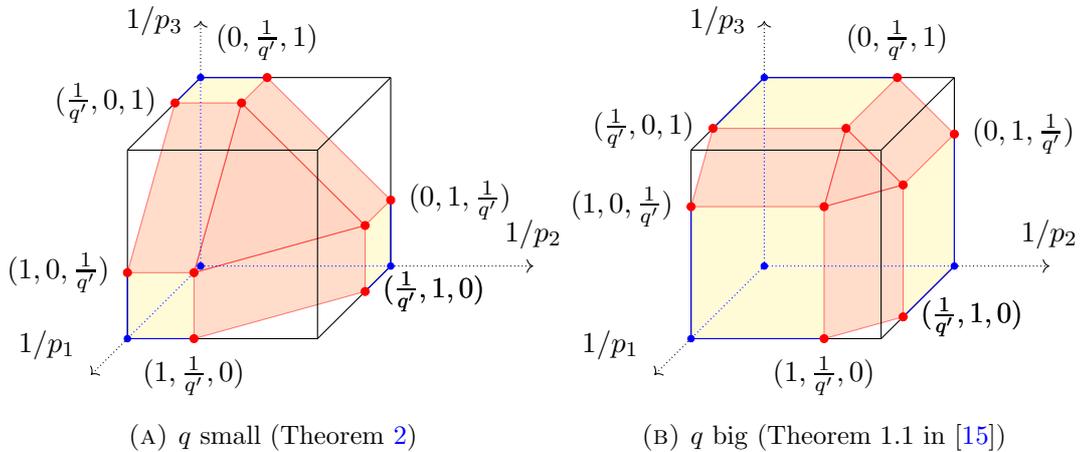
\begin{figure}[h]
	\newcommand{\Depth}{2.5}
	\newcommand{\Height}{2.5}
	\newcommand{\Width}{2.5}
	
	\newcommand{\qsmall}{0.35}
	\newcommand{\qbig}{0.7}

\centering
\begin{subfigure}{.5\textwidth}
	\centering
	\begin{tikzpicture}
\coordinate (O) at (0,0,0);
\coordinate (A) at (0,0,\Height);
\coordinate (B) at (\Depth,0,0);
\coordinate (C) at (0,\Width,0);
\coordinate (D) at (0,\Width,\Height);
\coordinate (E) at (\Depth,0,\Height);
\coordinate (F) at (\Depth,\Width,0);
\coordinate (G) at (\Depth,\Width,\Height);
\coordinate (hx) at (0,0,\qsmall*\Height);
\coordinate (hy) at (\qsmall*\Depth,0,0);
\coordinate (hz) at (0,\qsmall*\Width,0);
\coordinate (hxY) at (\Depth,0,\qsmall*\Height);
\coordinate (hyX) at (\qsmall*\Depth,0,\Height);
\coordinate (hzY) at (\Depth,\qsmall*\Width,0);
\coordinate (hxZ) at (0,\Width,\qsmall*\Height);
\coordinate (hyZ) at (\qsmall*\Depth,\Width,0);
\coordinate (hzX) at (0,\qsmall*\Width,\Height);
\coordinate (hxhyZ) at (\qsmall*\Depth,\Width,\qsmall*\Height);	
\coordinate (hyhzX) at (\qsmall*\Depth,\qsmall*\Width,\Height);	
\coordinate (hzhxY) at (\Depth,\qsmall*\Width,\qsmall*\Height);	

\coordinate (X) at (0,0,1.5*\Height);
\coordinate (Y) at (1.75*\Depth,0,0);
\coordinate (Z) at (0,1.3*\Width,0);

\draw[white,fill=yellow!35] (C) -- (hxZ) -- (hxhyZ) -- (hyZ) -- cycle;
\draw[white,fill=yellow!35] (A) -- (hyX) -- (hyhzX) -- (hzX) -- cycle;
\draw[white,fill=yellow!35] (B) -- (hzY) -- (hzhxY) -- (hxY) -- cycle;
\draw[white,fill=yellow!20] (O) -- (A) -- (hyX) -- (hxY) -- (B) -- cycle;
\draw[white,fill=yellow!20] (O) -- (B) -- (hzY) -- (hyZ) -- (C) -- cycle;
\draw[white,fill=yellow!20] (O) -- (A) -- (hzX) -- (hxZ) -- (C) -- cycle;
\draw[red, fill=red!30,opacity=0.5] (hyhzX) -- (hzhxY) -- (hxhyZ) -- cycle; 
\draw[red, fill=red!25,opacity=0.5] (hyhzX) -- (hzhxY) -- (hxY)-- (hyX) -- cycle; 
\draw[red, fill=red!25,opacity=0.5] (hyhzX) -- (hxhyZ) -- (hxZ)-- (hzX) -- cycle; 
\draw[red, fill=red!25,opacity=0.5] (hzhxY) -- (hxhyZ) -- (hyZ)-- (hzY) -- cycle; 


\draw[black] (G) -- (D);	
\draw[black] (G) -- (E);	
\draw[black] (G) -- (F);	
\draw[black] (A) -- (D);	
\draw[black] (C) -- (D);	
\draw[black] (F) -- (C);
\draw[black] (F) -- (B);
\draw[black] (E) -- (A);
\draw[black] (E) -- (B);

\draw[blue,densely dotted] (O) -- (A);
\draw[blue,densely dotted] (O) -- (B);
\draw[blue,densely dotted] (O) -- (C);
\draw[blue] (A) -- (hyX);
\draw[blue] (A) -- (hzX);
\draw[blue] (B) -- (hxY);
\draw[blue] (B) -- (hzY);
\draw[blue] (C) -- (hxZ);
\draw[blue] (C) -- (hyZ);


\foreach \rr in {hxY, hxZ, hyX, hyZ, hzX, hzY, hxhyZ, hyhzX, hzhxY} {\filldraw[red](\rr) circle(1.5pt);}
\foreach \rr in {O, A, B, C} {\filldraw[blue](\rr) circle(1.2pt);}
\node [below =1mm of hyX]  {$(1,\frac1{q'},0)$};
\node [left = 1mm of hzX]  {$(1,0,\frac1{q'})$};
\node [right =1mm of hxY]  {$(\frac1{q'},1,0)$};
\node [right =1mm of hzY]  {$(0,1,\frac1{q'})$};
\node [above =1mm of hyZ]  {$(0,\frac1{q'},1)$};
\node [left =1mm of hxZ]  {$(\frac1{q'},0,1)$};
\node [right =1mm of hxY]  {$(\frac1{q'},1,0)$};


\node [above left =1mm of X]  {$1/p_1$};
\node [above=1mm of Y] {$1/p_2$};
\node [left=1mm of Z] {$1/p_3$};
\draw[->,black,densely dotted] (A) -- (X);
\draw[->,black,densely dotted] (B) -- (Y);
\draw[->,black,densely dotted] (C) -- (Z);

	\end{tikzpicture}
	\caption{$q$ small (\Cref{th:2})}
	\label{fig:sub1}
\end{subfigure}%
\begin{subfigure}{.5\textwidth}
	\centering
	\begin{tikzpicture}
	\coordinate (O) at (0,0,0);
	\coordinate (A) at (0,0,\Height);
	\coordinate (B) at (\Depth,0,0);
	\coordinate (C) at (0,\Width,0);
	\coordinate (D) at (0,\Width,\Height);
	\coordinate (E) at (\Depth,0,\Height);
	\coordinate (F) at (\Depth,\Width,0);
	\coordinate (G) at (\Depth,\Width,\Height);
	\coordinate (hx) at (0,0,\qbig*\Height);
	\coordinate (hy) at (\qbig*\Depth,0,0);
	\coordinate (hz) at (0,\qbig*\Width,0);
	\coordinate (hxY) at (\Depth,0,\qbig*\Height);
	\coordinate (hyX) at (\qbig*\Depth,0,\Height);
	\coordinate (hzY) at (\Depth,\qbig*\Width,0);
	\coordinate (hxZ) at (0,\Width,\qbig*\Height);
	\coordinate (hyZ) at (\qbig*\Depth,\Width,0);
	\coordinate (hzX) at (0,\qbig*\Width,\Height);
	\coordinate (hxhyZ) at (\qbig*\Depth,\Width,\qbig*\Height);	
	\coordinate (hyhzX) at (\qbig*\Depth,\qbig*\Width,\Height);	
	\coordinate (hzhxY) at (\Depth,\qbig*\Width,\qbig*\Height);	
	
	\coordinate (X) at (0,0,1.5*\Height);
	\coordinate (Y) at (1.5*\Depth,0,0);
	\coordinate (Z) at (0,1.3*\Width,0);

	\draw[white,fill=yellow!35] (C) -- (hxZ) -- (hxhyZ) -- (hyZ) -- cycle;
	\draw[white,fill=yellow!35] (A) -- (hyX) -- (hyhzX) -- (hzX) -- cycle;
	\draw[white,fill=yellow!35] (B) -- (hzY) -- (hzhxY) -- (hxY) -- cycle;
	\draw[white,fill=yellow!20] (O) -- (A) -- (hyX) -- (hxY) -- (B) -- cycle;
	\draw[white,fill=yellow!20] (O) -- (B) -- (hzY) -- (hyZ) -- (C) -- cycle;
	\draw[white,fill=yellow!20] (O) -- (A) -- (hzX) -- (hxZ) -- (C) -- cycle;
	\draw[red, fill=red!30,opacity=0.5] (hyhzX) -- (hzhxY) -- (hxhyZ) -- cycle; 
	\draw[red, fill=red!25,opacity=0.5] (hyhzX) -- (hzhxY) -- (hxY)-- (hyX) -- cycle; 
	\draw[red, fill=red!25,opacity=0.5] (hyhzX) -- (hxhyZ) -- (hxZ)-- (hzX) -- cycle; 
	\draw[red, fill=red!25,opacity=0.5] (hzhxY) -- (hxhyZ) -- (hyZ)-- (hzY) -- cycle; 
	
	
	\draw[black] (G) -- (D);	
	\draw[black] (G) -- (E);	
	\draw[black] (G) -- (F);	
	\draw[black] (A) -- (D);	
	\draw[black] (C) -- (D);	
	\draw[black] (F) -- (C);
	\draw[black] (F) -- (B);
	\draw[black] (E) -- (A);
	\draw[black] (E) -- (B);

	\draw[blue,densely dotted] (O) -- (A);
	\draw[blue,densely dotted] (O) -- (B);
	\draw[blue,densely dotted] (O) -- (C);
	\draw[blue] (A) -- (hyX);
	\draw[blue] (A) -- (hzX);
	\draw[blue] (B) -- (hxY);
	\draw[blue] (B) -- (hzY);
	\draw[blue] (C) -- (hxZ);
	\draw[blue] (C) -- (hyZ);


	\foreach \rr in {hxY, hxZ, hyX, hyZ, hzX, hzY, hxhyZ, hyhzX, hzhxY} {\filldraw[red](\rr) circle(1.5pt);}
	\foreach \rr in {O, A, B, C} {\filldraw[blue](\rr) circle(1.2pt);}
	\node [below =1mm of hyX]  {$(1,\frac1{q'},0)$};
	\node [left = 1mm of hzX]  {$(1,0,\frac1{q'})$};
	\node [right =1mm of hxY]  {$(\frac1{q'},1,0)$};
	\node [right =1mm of hzY]  {$(0,1,\frac1{q'})$};
	\node [above =1mm of hyZ]  {$(0,\frac1{q'},1)$};
	\node [left =1mm of hxZ]  {$(\frac1{q'},0,1)$};
	\node [right =1mm of hxY]  {$(\frac1{q'},1,0)$};
	
	
	\node [above left =1mm of X]  {$1/p_1$};
	\node [above=1mm of Y] {$1/p_2$};
	\node [left=1mm of Z] {$1/p_3$};
	\draw[->,black,densely dotted] (A) -- (X);
	\draw[->,black,densely dotted] (B) -- (Y);
	\draw[->,black,densely dotted] (C) -- (Z);

\end{tikzpicture}
\caption{$q$ big (Theorem 1.1 in \cite{GHHP22})}
	\label{fig:sub2}
\end{subfigure}
\caption{Optimal range of $L^{p_1}(\R^n)\times L^{p_2}(\R^n)\times L^{p_3}(\R^n)\to L^p(\R^n)$ boundedness of $T^3_\Omega$ when $\Omega\in L^q(\S^{3n-1})$, $q>1$.}
\label{fig:2}
\end{figure}
 \Cref{th:2} provides the optimal range of exponents such that the strong--type boundedness holds.  Building on a counterexample introduced in \cite{DGHST11} and refined in \cite{GHS19}, it was proven in \cite{GHHP22} that $\mathcal{H}^m(q)$ is the largest open set of exponents in which $T^m_\Omega$ is bounded. The counterexample in the next proposition also deals with the endpoints in $\mathcal{H}^m(q)$ for which we have equality in one or more of the equations in \eqref{eq:H^m}.

\begin{prop} \label{pr:counter} If $q>1$, $1<p_1,\dots,p_m<\infty$, $\frac{1}{p}=\sum_{i=1}^m \frac{1}{p_i}$ and $\left(\frac{1}{p_1},\dots,\frac{1}{p_m}\right)\not\in \mathcal{H}^m(q)$, then there exist an odd function $\Omega\in L^q(\S^{mn-1})$ and $f_i \in L^{p_i}(\R^n)$, $i=1,\dots,m$, such that $T_\Omega^m(f_1,\dots,f_m)\not\in L^{p}(\R^n)$.
\end{prop}

In Section~\ref{S:reduction}, we use the bootstrap argument introduced in \cite{GHHP22} to reduce Theorem~\ref{th:2} to proving an estimate of minimal blow-up in the Banach range of exponents when $\Omega \in L^1(\mathbb S^{mn-1})$. This estimate is stated in Proposition~\ref{pr:L1} below and will be proved in Section~\ref{S:proof-L1}. \Cref{section:counter} contains the proof of \Cref{pr:counter}.

\medskip
 \textit{Acknowledgment.} We thank the anonymous referee for their valuable comments that helped to improve the exposition. We also thank Bae Jun Park for pointing out to us an error in the statement of \Cref{pr:L1} in an earlier version of the manuscript.

\section{Reductions}\label{S:reduction}

We fix $m\in\N$, $q\geq 1$ and $\Omega\in L^q(\S^{mn-1})$ with mean value zero and in the following we simplify notation writing $T$ instead of $T^m_\Omega$. Let $\eta$ be a Schwartz function in $\R^{mn}$ such that $\wh\eta(\xi) = 1$ when $|\xi|\leq 1$ and $\wh\eta(\xi) =0$ when $|\xi|\geq 2$, and let $\wh\beta = \wh\eta - \wh\eta(2\cdot)$. We decompose the kernel $K$ by setting $K^i(y) = K(y)\wh{\beta}(2^{-i}y)$ for $i\in \Z$ and 
\[K^i_j (y) = \left({ \wh{K^i}(\cdot)} \wh\beta({2^{-j+i}}\cdot)\right)^\vee(y), \qquad K_j = \sum_{i=-\infty}^{\infty} K^i_j, \qquad j\in \Z.
\]
We decompose $T = \sum_{i=-\infty}^{\infty} \sum_{j=-\infty}^\infty T^i_j = \sum_{j=-\infty}^{\infty} T_j$ accordingly, where $K^i_j$ is the kernel of $T^i_j$ and $K_j$ the kernel of $T_j$. We note that 
\[
\widehat{K_j}=\sum_{i=-\infty}^{\infty} \widehat{K_j^0}(2^i\cdot),
\]
by the homogeneity of $K$. 

Let $1<p_1,p_2,\dots,p_m<\infty$ and let $p$ satisfy $\frac{1}{p} = \sum_{i=1}^m \frac{1}{p_i}$. 
We obtain \Cref{th:2} by proving bounds of the form
\begin{equation}\label{E:bound-Tj}
\|T_j\|_{\MLp} \lesssim Q(j) \|\Omega\|_{L^q(\S^{mn-1})}
\end{equation}
that are summable in $j$. Here and in the following we are using the usual notation $A\lesssim B$ to denote that there exists a constant $c$ such that $A\leq c B$ and we use $A\lesssim_\gamma B$ to imply that $c$ depends on the parameter $\gamma$.

We first recall two estimates of the type~\eqref{E:bound-Tj} which are available in the literature.
The following proposition follows from the multilinear Calder\'on-Zygmund theory~\cite{GT99}. Its specific version stated below was proved in~\cite[Proposition 3]{GHH18} in the bilinear case, and the same argument translates also to the multilinear setting, as observed in~\cite[Equation (3.3)]{GHHP22}.

\begin{customprop}{A}[\cite{GHH18, GHHP22}]\label{pr:GT} Let $q>1$ and $0<\delta<1/q'$. Then inequality~\eqref{E:bound-Tj} holds with $Q(j) = 2^{(mn-\d)j}$ if $j\geq 0$ and $Q(j) = 2^{-|j|(1-\d)}$ if $j<0$.
\end{customprop}

Let $\mathcal{C}^m(q) = \big(0,\frac{1}{q'}\big]^m$.  The following proposition follows from the results in~\cite[Section 3]{GHHP22}. 

\begin{customprop}{B}[\cite{GHHP22}\label{pr:locL2}] 
		Let $q \geq 2$ and $\vec{\frac{1}{p}}\in \mathcal{C}^m(q)$. Then there exist $r = r(\vec{\frac{1}p},n)>0$ and $j_0=j_0(m,n) \in \Z$ such that inequality~\eqref{E:bound-Tj} holds with $Q(j) = 2^{-rj}$ if $j \geq j_0$. 
\end{customprop}

In what follows, we fix $j_0$ from Proposition~\ref{pr:locL2} and we only consider $j\geq j_0$ in light of Proposition~\ref{pr:GT}. The following proposition will be proved in Section~\ref{S:proof-L1}.

\begin{prop}\label{pr:L1} 
Let $j\geq j_0$, $p>1$, $q=1$ and $\varepsilon>0$. Then inequality~\eqref{E:bound-Tj} holds with $Q(j)  \lesssim_\varepsilon 2^{j\varepsilon}$.
\end{prop}

We assume \Cref{pr:L1} for now and we state and prove several claims that will  gradually improve the dependence of the bound in estimate~\eqref{E:bound-Tj} on $j$. Theorem~\ref{th:2} will then be a direct consequence of Proposition~\ref{pr:GT} and Claim~\ref{c:3} below.

\begin{claim}\label{c:1}
Let $j\geq j_0$, $q>1$ and $\vec{\frac{1}{p}}\in \mathcal{C}^m(q)$. Then there exists an $r = r(\vec{\frac{1}p},q,n)>0$ such that inequality~\eqref{E:bound-Tj} holds with $Q(j) = 2^{-rj}$.
\end{claim}

\begin{proof}
If $q \geq 2$ then the claim follows from Proposition~\ref{pr:locL2}. To prove the claim in the case $1<q<2$, we will use complex interpolation between the estimates in Proposition~\ref{pr:locL2} and Proposition~\ref{pr:L1}. To be able to proceed with the interpolation argument, we view $T_j$ as an $(m+1)$--linear operator. More precisely, let  
\[\tilde{T}(\Omega, f_1,\dots, f_m)(x) := pv \int_{\R^{mn}} \left(\Omega - Avg(\Omega) \right)(y')|y|^{-mn} \prod_{i=1}^m f_i(x-y_i) dy,
\]
where $Avg(\Omega) = \int_{\S^{mn-1}} \Omega d\s$ is the mean value of $\Omega$. Note that $\tilde{T}$ is $(m+1)$--linear and for every $\Omega$ with mean zero and for any Schwartz functions $f_1,\dots,f_m$, it agrees with the singular integral defined in \eqref{eq:T}. Therefore, it is given by an $m$--linear convolution with a distribution $W_\Omega$ and $W_{\Omega - Avg(\Omega) } = W_\Omega$, with equality holding in the sense of distributions. 

Similar to $T_j$, we define $\tilde{T}_j$ by setting $\tilde{K}_\Omega = \frac{ \Omega({y'}) - Avg(\Omega)}{|y|^{mn}}$,  $\tilde{K}_\Omega^i(y) = \tilde{K}_\Omega(y)\wh{\beta}(2^{-i}y)$, 
\[\left(\tilde{K}_\Omega\right)^i_j (y) = \left({ \wh{\tilde{K}_\Omega^i}(\cdot)} \wh\beta({2^{-j+i}}\cdot)\right)^\vee(y), \qquad \left(\tilde{K}_\Omega\right)_j = \sum_{i=-\infty}^{\infty} \left(\tilde{K}_\Omega\right)^i_j,
\]
and 
\[\tilde{T}_j(\Omega, f_1,\dots, f_m)(x) =  pv \int_{\R^{mn}} \left(\tilde{K}_\Omega\right)_j(y) \prod_{i=1}^m f_i(x-y_i) dy.
\]
It is immediate that the bounds described above for $T_j$ also hold for $\tilde{T}_j$ and we can use $(m+1)$--linear interpolation (see \cite[Lemma 2.1 and Proposition 2.2]{BOS2009} for the multilinear interpolation result used here) between Banach space $(m+1)$--tuplets of the form $L^q(\S^{mn-1})$, $L^{p_1}(\R^n), \dots, L^{p_m}(\R^n)$.

Fix a $1<q<2$ and a $\vec{\frac{1}{p}}\in \mathcal{C}^{m}(q)$ and assume that $p_1=\min\{p_1,\dots,p_m\}$ (without a loss of generality, since the other cases are symmetric) and note that $p_1>2$ since $q<2$. Then $\frac{p_1}{2}\vec{\frac{1}p} = \left(\frac12,\frac12\frac{p_1}{p_2},\dots, \frac12\frac{p_1}{p_m}\right)\in\mathcal{C}^m(2)$ and Proposition~\ref{pr:locL2} yields
\begin{equation}\label{eq:IntA}
	\|\tilde{T}_j\|_{{L^2(\mathbb S^{mn-1})\times} L^2(\R^n)\times L^{\frac{2p_2}{p_1}}(\R^n)\times\cdots\times L^{\frac{2p_m}{p_1}}(\R^n)\to L^{\frac{2p}{p_1}}(\R^n)} \lesssim 2^{-r_0 j} 
\end{equation}
for some positive exponent $r_0$.

Let $\varepsilon>0$ be small, to be determined momentarily. If $\delta \in (0,1)$ satisfies $\delta<p$, then \Cref{pr:L1} and the embedding $L^{\frac{1}{1-\delta}}(\S^{mn-1})\hookrightarrow L^1(\S^{mn-1})$ imply that there is a constant $C_\varepsilon$ such that  
\begin{equation}\label{eq:IntB}
	\|\tilde{T}_j\|_{{L^{\frac1{1-\delta}}(\mathbb S^{mn-1})\times} L^{\frac{p_1}{\delta}}(\R^n)\times L^{\frac{p_2}{\delta}}(\R^n)\times\cdots\times L^{\frac{p_m}{\delta}}(\R^n)\to L^{\frac{p}{\delta}}(\R^n)} \lesssim C_\varepsilon 2^{\varepsilon j} 
\end{equation}
Interpolating between the estimates \eqref{eq:IntA} and \eqref{eq:IntB} we obtain 
\begin{equation}\label{eq:IntC}
	\|\tilde{T}_j\|_{{ L^q(\mathbb S^{mn-1})\times }\MLp} \lesssim C_\varepsilon^{\frac{p_1 -2}{p_1-2\delta}} (2^{j})^{\varepsilon\frac{p_1 -2}{p_1-2\delta} - r_0\frac{2-2\delta}{p_1-2\delta} }.
\end{equation}
Finally, choosing $\varepsilon< r_0 \frac{2-2\delta}{p_1-2}$ completes the proof of the claim.
\end{proof}

\begin{claim}\label{c:2} 
Let $j\geq j_0$,  $q>1$, $\vec{\frac{1}{p}}\in \mathcal{H}^m(q)$ and $\varepsilon>0$. Then inequality~\eqref{E:bound-Tj} holds with $Q(j)  \lesssim_\varepsilon 2^{j\varepsilon}$.
\end{claim}

\begin{proof}
Claim~\ref{c:1} yields that~\eqref{E:bound-Tj} holds with $Q(j)  \lesssim_\varepsilon 2^{j\varepsilon}$ whenever $\vec{\frac{1}{p}}\in \mathcal{C}^m(q)$. The proof then follows from the extension result of~\cite[Proposition 6.2]{GHHP22}.
\end{proof}

\begin{claim}\label{c:3} 
Let $j\geq j_0$, $\vec{\frac{1}{p}}\in \mathcal{H}^m(q)$ and $q>1$. Then there exists an $r = r(\vec{\frac{1}p}, q,n)>0$ such that inequality~\eqref{E:bound-Tj} holds with $Q(j)  \lesssim 2^{-rj}$.
\end{claim}

\begin{proof}
We observe that the point $\left(\frac{1}{q'},\frac{1}{q'},\dots,\frac{1}{q'}\right)$ is in the interior of $\mathcal{H}^m(q)$ and interpolate between Claim~\ref{c:1} and Claim~\ref{c:2} using~\cite[Theorem 7.2.2]{GModern}.
\end{proof}

\section{Proof of \Cref{pr:L1}}\label{S:proof-L1}

{ A critical tool for the proof of Proposition~\ref{pr:L1} is a bound for the shifted maximal and square functions. Let $\varphi$ and $\psi$ be Schwartz functions on $\R^n$ with compact Fourier supports such that $\wh{\psi}$ is supported away from the origin. For $t>0$ and $v\in\R^n$, we denote $\varphi_{t}^{v}(x) = t^n\varphi(t x-v)$ 
and similarly for $\psi_{t}^{v}(x)$.


%
With this notation, we have 
\begin{customprop}{C}[\cite{GModern,M14}\label{pr:Mus}] 
Let $1<p<\infty$. Then
\begin{align*}
&\left\|\sup_{r\in\Z} \big|f\ast \varphi^v_{2^r}\big| \right\|_{L^p(\R^n)}\lesssim \log(2+|v|)\left\|f\right\|_{L^p(\R^n)}\\
\text{and }\quad &\bigg\|\Big(\sum_{r\in\Z} \left|f\ast \psi^v_{2^r}\right|^2\Big)^{\frac12}\bigg\|_{L^p(\R^n)}\lesssim \log(2+|v|)\left\|f\right\|_{L^p(\R^n)}.
\end{align*}
\end{customprop}

In the one-dimensional setting, Proposition\cref{pr:Mus} is proven in \cite[Theorems 4.1, 5.1]{M14}. The first of the two inequalities is also contained in~\cite[Chapter II, 5.10]{SHarmonic}. The proof in the $n$-dimensional case can be found in
\cite[Proposition 7.5.1]{GModern}.

We next apply Proposition\cref{pr:Mus} to prove Proposition~\ref{pr:L1}. Our argument follows the ideas from \cite{M14,M14ii}.}

\begin{proof}[Proof of \Cref{pr:L1}]
Given $j\in \Z$, let $\Lambda_j$ be the $(m+1)$-linear form given by
\[
\Lambda_j(f_1,\dots,f_{m+1})=\int_{\R^n} T_j(f_1,\dots,f_m)(x) f_{m+1}(x)\,dx.
\]
Since $p{>} 1$, we observe that Proposition~\ref{pr:L1} will follow if we prove that
\[
|\Lambda_j(f_1,\dots,f_{m+1})| {\lesssim} j^{{m}} \|\Omega\|_{L^1(\S^{mn-1})} \|f_{m+1}\|_{L^{p'}(\R^n)} \prod_{j=1}^m \|f_i\|_{L^{p_i}(\R^n)}
\]
for $j\geq j_0$.
To simplify notation, we write $\vec{x}=(x_1,\dots,x_m)\in\R^{mn}$ for $x_i\in\R^n$. We rewrite the form $\Lambda_j$ as follows:
	\begin{align*} &\Lambda_j (f_1,\dots,f_{m+1})\\
		&= \sum_{k\in \Z} \int_{\R^n} \int_{\R^{mn}} \wh {K_{j}^0}(2^k\vec{\xi}\,)  \left(\prod_{i=1}^m \wh f_i(\xi_i)\right)  f_{m+1}(x) e^{2\pi i x\cdot(\xi_1+\cdots+\xi_m)} d\vec{\xi}dx \\
		&= \sum_{k\in \Z} \int_{\R^{mn}} \wh {K_{j}^0}(2^k\vec{\xi}\,)  \left(\prod_{i=1}^m \wh f_i(\xi_i)\right)  \wh{f_{m+1}}(-\xi_1-\cdots-\xi_m) d\vec{\xi}.
	\end{align*}
Notice that $\wh {K_{j}^0}(2^k\cdot)$ is supported in the $mn$--dimensional annulus where $|\vec{\xi}|\sim 2^{j-k}$. Therefore, at least two of the vectors $\xi_1, \dots, \xi_m, (-\xi_1-\cdots - \xi_m)$ have to belong to an $n$--dimensional annulus of radius about $2^{j-k}$. 
By a smooth partition of unity, it suffices to consider pieces of the form 
\begin{equation*}
	\sum_{k\in \Z} \int_{\R^{mn}} \wh {K_{j}^0}(2^k\vec{\xi}\,)  \left(\prod_{i=1}^m \wh f_i(\xi_i)\wh{\phi_i}(2^{k-j}\xi_i) \right) \wh{f_{m+1}}(-\xi_1-\cdots-\xi_m) \wh{\phi_{m+1}}(2^{k-j}(-\xi_1-\cdots-\xi_m))d\vec{\xi},	
\end{equation*}
where $\phi_i$, $i=1,\dots,m+1$, are {Schwartz functions with compact Fourier support such that at least two of them are Fourier supported away from the origin. We denote these two functions by $\phi_{l}$ and $\phi_{l'}$.} 
Using Fourier inversion, the form can then be written as
\begin{align*}& \sum_{k\in \Z} \int_{\R^n} \int_{\R^{mn}} 2^{-kmn} {K_{j}^0}(2^{-k}\vec{y}\,)  \left(\prod_{i=1}^m  f_i\ast (\phi_i)_{2^{j-k}}(x-y_i)\right)  f_{m+1}\ast(\phi_{m+1})_{2^{j-k}}(x) d\vec{y}dx \\
= &\sum_{k\in \Z} \int_{\R^n} \int_{\R^{mn}}  {K_{j}^0}(\vec{y}\,) \left( \prod_{i=1}^m  f_i\ast (\phi_i)_{2^{j-k}}(x-2^{k}y_i)\right)  f_{m+1}\ast(\phi_{m+1})_{2^{j-k}}(x) d\vec{y}dx \\
= &\sum_{k\in \Z} \int_{\R^n} \int_{\R^{mn}}  {K_{j}^0}(\vec{y}\,)  \prod_{i=1}^{m+1}  f_i\ast (\phi_i)^{2^jy_i }_{2^{j-k}}(x)   d\vec{y}dx,
\end{align*}
where we have set $y_{m+1} = 0$. We also formally set $p_{m+1}=p'$. For a fixed $\vec{y}\in\R^{mn}$, we estimate
\begin{align*}
&\left|\sum_{k\in \Z}\int_{\R^n} \prod_{i=1}^{m+1} f_i\ast (\phi_i)^{2^jy_i }_{2^{j-k}}(x) dx  \right|\\
&\leq \int_{\R^n} \prod_{i\neq l,l'}  \sup_{k\in\Z} | f_i\ast (\phi_i)^{2^jy_i }_{2^{j-k}}(x)| \prod_{r\in\{l,l'\}} \left(\sum_{k\in \Z}| f_r\ast (\phi_r)^{2^jy_r }_{2^{j-k}}(x)|^2 \right)^{1/2} dx\\
&\lesssim \prod_{i\neq l,l'}\left\| \sup_{k\in\Z} | f_i\ast (\phi_i)^{2^jy_i }_{2^{j-k}}|\right\|_{L^{p_i}(\R^n)}\prod_{r\in\{l,l'\}} \left\|\left(\sum_{k\in \Z}| f_r\ast (\phi_r)^{2^jy_r }_{2^{j-k}}|^2 \right)^{1/2}\right\|_{L^{p_r}(\R^n)}\\
&\lesssim \left(\prod_{i=1}^m { \log(2+ 2^j|y_i|)} \|f_i\|_{L^{p_i}(\R^n)}\right)\|f_{m+1}\|_{L^{p'}(\R^n)}\\
&\lesssim j^{{m}} \left(\prod_{i=1}^m { \log(2+ |y_i|)} \|f_i\|_{L^{p_i}(\R^n)}\right)\|f_{m+1}\|_{L^{p'}(\R^n)},
\end{align*}
using Proposition\cref{pr:Mus}. Therefore, it suffices to show that 
	\begin{align*} &\int_{\R^{mn}}  \left| K_{j}^0(\vec{y})\right| \prod_{i=1}^m { \log(2+|y_i|)}d\vec{y} \lesssim \|\Omega\|_{L^1(\S^{mn-1})} .
	\end{align*}

	We  recall that $K_j^0=K^0 \ast \beta_{2^j}$. Since $\beta$ is a Schwartz function, we may estimate it as
	\[
	|\beta(x)| {\lesssim} \sum_{k=1}^\infty 2^{-2mnk} \chi_{\B^{mn}(0,2^k)}(x).
	\]
	We further notice that the function $|K^0|\ast \chi_{\B^{mn}(0,2^{k-j})}$ is supported in the set $\B^{mn}(0,2^{k+1})$. Therefore,
	\begin{align*}
		& \int_{\R^{mn}}  \left| K_{j}^0(\vec{y})\right| \prod_{i=1}^m{\log(2+|y_i|)}d\vec{y}\\
		&\lesssim \sum_{k=1}^\infty 2^{mnj-2mnk} \int_{\R^{mn}} |K^0| \ast \chi_{\B^{mn}(0,2^{k-j})} (\vec{y}) \prod_{i=1}^m { \log(2+|y_i|)} d\vec{y}\\
		&\lesssim \sum_{k=1}^\infty k^{{ m}} 2^{mnj-2mnk} \||K^0| \ast \chi_{\B^{mn}(0,2^{k-j})}\|_{L^1(\R^{mn})} 
		\lesssim \|K^0\|_{L^1(\R^{mn})}  \sum_{k=1}^\infty k^{{ m}} 2^{-mnk}\\
		&\lesssim\|K^0\|_{L^1(\R^{mn})} 
		\lesssim \|\Omega\|_{L^1(\S^{mn-1})},
	\end{align*}
	as desired.%
\end{proof}

\section{Counterexample}\label{section:counter}
\begin{proof}[Proof of \Cref{pr:counter}] 
	Since $(\frac{1}{p_1},\dots,\frac{1}{p_m}) \notin \mathcal{H}^m(q)$, there is a multiindex $\alpha$ such that~\eqref{eq:H^m} is not satisfied.
	Without a loss of generality, we can assume that $\alpha=(1,\dots,1,0,\dots,0)$, with $|\alpha|= \kappa\geq 2$, since the cases $|\alpha|=0$ and $|\alpha|=1$ are trivial, and the operator is unaffected by permutations of the functions $f_i$. 
	
	For $i=1,\dots,m$, let 
	\begin{equation*}
		\begin{cases}
			f_i(x) &= |x|^{-n/p_i}\big|\log{|x|}\big|^{-\gamma/p_i}\chi_{|x|<\frac12}(x), \quad \text{for } i=1,\dots,\kappa,\\
			f_i(x) &= |x|^{-n/p_i}\big|\log{|x|}\big|^{-\gamma/p_i}\chi_{|x|>2}(x), \quad \text{for } i=\kappa + 1,\dots,m,
		\end{cases}
	\end{equation*}
	for some $\gamma>1$ to be determined, and note that $f_i\in L^{p_i}(\R^n)$. 
	For $\theta=(\theta_1,\dots,\theta_{m}) \in \S^{mn-1}$, we define 
	\[\Omega(\theta)=\operatorname{sgn}(\theta^1_{1}) \chi_{\mathbb B^{\kappa n}\left(0,\frac1{4\sqrt{\kappa}}\right)} (\theta_1,\dots,\theta_\kappa)\left(\sum_{i=2}^\kappa|\theta_1-\theta_i|^2\right)^{-\frac{(\kappa-1)n}{2q}} \left|\log{\sum_{i=2}^\kappa|\theta_1-\theta_i|^2}\right|^{-\frac{\gamma}{q}},
	\]
	where $\theta_{1}^1$ denotes the first coordinate of $\theta_1$. 
	%
	%
	%
	Then $\Omega$ is an odd function in $L^{q}(\S^{mn-1})$. Indeed, using the co-area formula, \cite[Appendix D]{GClassical}, and setting $r_{\theta,\kappa} = \sqrt{1-\sum_{i=1}^k|\theta_i|^2}$, we have 
	\begin{align*}&\int_{\S^{mn-1}} |\Omega(\theta)|^{q} d\theta = \int_{\mathbb B^{\kappa n}} \int_{r_{\theta,\kappa}\S^{(m-\kappa)n-1}} |\Omega(\theta)|^{q} \frac{d\s(\theta_{\kappa+1},\dots,\theta_m)}{r_{\theta,\kappa}} d\theta_1\cdots d\theta_\kappa.
	\end{align*}
	Since $1-\sum_{i=1}^k|\theta_i|^2\approx 1$ when $\theta$ belongs to the support of $\Omega$, the previous expression is bounded by
	\begin{align*}&\int_{\mathbb B^{\kappa n}(0,\frac1{4\sqrt{\kappa}} )} \left(\sum_{i=2}^\kappa|\theta_1-\theta_i|^2\right)^{-\frac{(\kappa-1)n}{2}} \left|\log{\sum_{i=2}^\kappa|\theta_1-\theta_i|^2}\right|^{-\gamma} d\theta_1\cdots d\theta_\kappa\\
		&\lesssim\int_{\mathbb B^{(\kappa-1) n}(0,\frac1{4\sqrt{\kappa}} )}|u|^{-{(\kappa-1)n}} \big|\log{|u|}\big|^{-{\gamma}}du_2\cdots du_\kappa <\infty,
	\end{align*}
	where the first inequality follows by the change of coordinates to $u_i = \theta_1 - \theta_i$ for $i=2,\dots,\kappa$.
	
	Let $x\in \R^{n}$ be a vector such that $x^1 >10$. We denote $\bar{x} = (x,x,\dots,x)\in(\R^n)^m$.
	Then 
	\begin{align*}T_\Omega^m(f_1,\dots, f_m)(x) &= \int\limits_{|y_1|,\dots,|y_\kappa|<\frac12}\hspace{-.15cm}\cdots\hspace{-.15cm}\int\limits_{|y_{\kappa+1}|,\dots,|y_m|>2} \left(\prod_{i=1}^m |y_i|^{-n/p_i}\left|\log|y_i|\right|^{-\gamma/p_i}\right)\frac{\Omega\left(\frac{\bar{x} -y}{|\bar{x} -y|}\right)}{|\bar{x} -y|^{mn}}dy.
	\end{align*}
	Note that, since $|y_1|<\frac12$,  the first coordinate of $(\bar{x}-y)_1$ is positive, and thus the integrand in the above integral is positive. 
	Therefore, we can estimate it from below by the integral over the smaller region $A\times B$, where
	\begin{equation*}
		A = \left\{(y_1,\dots,y_\kappa)\in\B^{\kappa n}\left(0,\frac1{|x|}\right)\,:\,|y_1|\approx \cdots|y_\kappa| \approx |y_1 - y_2|\approx \dots \approx |y_1-y_\kappa|\right\}
	\end{equation*}
	and
	\begin{equation*}
		B=\B^{n(m-\kappa)}\left((\underbrace{x,\dots,x}_{m-\kappa}),\frac{|x|}{2}\right).
	\end{equation*}
	We write $y=(z,z')$, where $z=(y_1,\dots,y_\kappa)$ and $z'=(y_{\kappa+1},\dots,y_m)$.  In $A\times B$, we have for $i=2,\dots \kappa$,
	\begin{equation*}
	 |y_1-y_i|\approx |z|<\frac1{|x|} \quad \text{and } |\bar{x}-y|\approx |z'|\approx |x|,
	\end{equation*}
	and therefore, for $y\in A\times B$,
	\begin{align*}\Omega\left(\frac{\bar{x}-y}{|\bar{x}-y|}\right) &  =  \left(\sum_{i=2}^\kappa\left|\frac{y_1- y_i}{|\bar{x}-y|}\right|^2\right)^{-\frac{(\kappa-1)n}{2q}} \left|\log{\sum_{i=2}^\kappa\left|\frac{y_1- y_i}{|\bar{x}-y|}\right|^2}\right|^{-\frac{\gamma}{q}}\\
	&\approx  \left(\frac{|z|}{|x|}\right)^{-\frac{(\kappa-1)n}q}\left|\log\frac{|z|}{|x|} \right|^{-\frac{\gamma}{q}}\\
	&\gtrsim |x|^{\frac{(\kappa-1)n}q} |z|^{-\frac{(\kappa-1)n}q}\big|\log|z| \big|^{-\frac{\gamma}{q}}.
	\end{align*}
	Setting $\frac{1}{p_\alpha} = \sum_{i=1}^\kappa \frac{1}{p_i}$ and $\frac{1}{\wh{p}_\alpha} = \sum_{i=\kappa+1}^m \frac{1}{p_i}$, we estimate the absolute value of $T^m_\Omega(f_1,\dots, f_m)(x)$ from below by
	\begin{align*}&\int_{B}\int_{A} \left(\prod_{i=1}^m |y_i|^{-n/p_i}\left|\log|y_i|\right|^{-\gamma/p_i}\right)\frac{\left|\Omega\left(\frac{\bar{x} -y}{|\bar{x} -y|}\right)\right|}{|\bar{x} -y|^{mn}} dy\\
		&\gtrsim  |x|^{-mn+\frac{(\kappa-1)n}q} \int_{B} |z'|^{-\frac{n}{\wh{p}_\alpha}}(\log |z'|)^{-\frac{\gamma}{\wh{p}_\alpha}}dz' \int_{A} |z|^{-n\left(\frac{1}{p_\alpha}+\frac{\kappa-1}q\right)}\big|\log{|z|}\big|^{- \frac{\gamma}{p_\alpha}-\frac{\gamma}{q}} dz\\
		&\gtrsim (\log |x|)^{-\frac{\gamma}{\wh{p}_\alpha}}|x|^{-mn  +\frac{(\kappa-1)n}q - \frac{n}{\wh{p}_\alpha} + (m-\kappa)n}  \int_{A} |z|^{-n\left(\frac{1}{p_\alpha}+\frac{\kappa-1}q\right)} \big|\log|z|\big|^{- \frac{\gamma}{p_\alpha}-\frac{\gamma}{q}} dz,
	\end{align*}
	which diverges when $\frac{1}{p_\alpha}+\frac{\kappa-1}q>\kappa = |\alpha|$. 
	
	When $\frac{1}{p_\alpha}+\frac{\kappa-1}q= \kappa$, the exponents in the previous expression simplify to 
	\begin{align*}&(\log |x|)^{-\frac{\gamma}{\wh{p}_\alpha}}|x|^{-\frac{n}p}  \int_{A} |z|^{-\kappa n}\big|\log{|z|}\big|^{- \frac{\gamma}{p_\alpha}-\frac{\gamma}{q}} dz,
	\end{align*}
	which is equivalent to 
	\[
	(\log |x|)^{1-\frac{\gamma}{p} - \frac{\gamma}q}|x|^{-\frac{n}p}.
	\]
	Therefore,
	\begin{align*}\int_{\R^{n}} |T^m_\Omega(f_1,\dots,f_m)(x)|^p dx &\gtrsim \int_{\{x\in \R^n:~x^1>10\}} |x|^{-n} \left(\log |x|\right)^{p-\gamma - \frac{\gamma p}q}dx,
	\end{align*}
	which is finite if and only if $p-\gamma - \frac{\gamma p}q<-1 \Leftrightarrow \gamma> \frac{p+1}{\frac{p}{q}+1}$. Since $q>1$, we can choose $ 1<\gamma<\frac{p+1}{\frac{p}{q}+1}$ in the definitions of $f_1,\dots, f_m,$ and $\Omega$ to obtain the counterexample.

\end{proof}

\section*{Statements and Declarations}
\subsubsection*{\textbf {Conflict of interest}}
On behalf of all authors, the corresponding author states that there is no conflict of
interest.
\subsubsection*{\textbf {Data availability}}
Data sharing not applicable to this article as no datasets were generated or analysed
during the current study.

\Addresses

\end{document}